\documentclass[10pt]{amsart}

\usepackage[english]{babel}

\usepackage{amssymb,amsthm}

\usepackage{libertine}

\usepackage{enumitem,xfrac}

\usepackage[dvipsnames]{xcolor}

\usepackage{hyperref}

\hypersetup{
  pdftitle={\title},
  pdfauthor={\author},
  pageanchor=true,
  hyperfootnotes=false,
  breaklinks,
  colorlinks,
  allcolors=NavyBlue
  }

\usepackage{graphicx}

\newtheorem{theorem}{Theorem}

\newtheorem{definition}{Definition}
\newtheorem*{example}{Example}
\newtheorem{lemma}{Lemma}

\newtheorem*{problem*}{Problem}

\newtheorem{proposition}{Proposition}

\usepackage{enumitem,xfrac}
\usepackage{subcaption}
\begin{document}

\title{On Suffridge polynomials}

\author{Jimmy Dillies}
\address{University of Georgia, Athens  GA, U.S.A.}
\email{jimmy.dillies@uga.edu}

\author{Dmitriy Dmitrishin}
\address{Odessa National Polytechnic University, Odessa, Ukraine}
\email{dmitrishin@opu.ua}

\author{Alex Stokolos}
\address{Georgia Southern University, Statesboro GA, U.S.A.}
\email{astokolos@georgiasouthern.edu}


\begin{abstract}
We consider some known and some new properties of the family of polynomials introduced by Ted Suffridge in 1969. 
We begin by giving a brief overview of their extremal properties in classic and more recent work.  
We also give a compact form for Suffridge polynomials which matches a general pattern discovered by Brandt.
Our approach allows us to find the coefficients which Brandt's result was not giving explicitly. 
This new presentation provides us the tools to obtain an estimate of the rate of approximation of the generalized Koebe functions by univalent polynomials. 
Furthermore, we consider the presentation of Suffridge polynomials in Robertson's form and find the suiting Robertson measure.
This suggests a new way to approximate step functions by continuous monotonic ones.
We then study the lack of robustness  of the univalency of these polynomials and suggest a new family of polynomials for which we conjecture the univalency of a subclass. 
Namely, we proved a quite surprising fact that extending the family by letting the discrete argument in the polynomial coefficients become continuous one does not increase the set of univalent polynomials. 
Only the initial polynomials are univalent.
In this new one parameter family generalizing the Suffridge polynomials, it is remarkable that the Suffridge polynomials are already extremal as they correspond to the choice of the paramater set to 1; moreover the complex Fej\'er polynomials correspond the choice of the parameter set to 0, and the choice of the paramater set to -1 corresponds the polynomials $z+(z^N/N)$. 
Remarkably, computer simulations seem to clearly indicate that the image of the unit disc under these new polynomial mapping is a simply-connected region bounded by a simple curve. 
This justifies the conjectural univalency of these polynomials for the whole range of the parameters.
\end{abstract}

\maketitle

\section{Introduction}
%%%%%%%%%%

Suffridge polynomials are a central object of study in the theory of univalent polynomials. 
For example, in the book by P. Duren~\cite{Du}, the ``reference'' for univalent functions,  the section on univalent polynomials is devoted almost exclusively to Suffridge polynomials.

These polynomials were introduced by Ted Suffridge in 1969~\cite{S} and are still essential to modern research because of their remarkable properties and because of the otherwise very limited number of interesting known examples of univalent polynomials.
They are defined as
% change notation to the original by Suffridge?
\[
S_{N,j}(z)= \sum_{k=1}^N \frac{N+1-k}{N} \frac{\sin\frac{kj\pi }{N+1}}{\sin\frac{j\pi}{N+1}} z^k, \qquad j=1,...,N; \; N=1,2,...
\]
and Suffridge (ibid.) showed that they are a good approximation for the Koebe function 
\[
K(z)=\frac{z}{(1-z)^2}.
\]

In this note, we present an overview of some of the main results regarding Suffridge polynomials (Section~\ref{sec:hist}).
We then offer the first closed formula describing their general form (Section~\ref{sec:closed}) and show how this formula matches a general pattern discovered by Brandt (which had the inconvenient of not readily providing suitable coefficients) in Section~\ref{sec:brandt}.
After that, we study extremal properties of Suffridge polynomials (Section~\ref{sec:extremal}) and how they 'approximate' conformal mappings in the unit disc (Section~\ref{sec:approx}).
We also discuss their Robertson measure in Section~\ref{sec:robertson}.
Finally, in the last part of this paper how our polynomials lack 'robustness' (Section \ref{sec:robust}) and suggest a new family (Section~\ref{sec:new}) that, we conjecture, palliates some of these weaknesses.

\section{A Tale of Suffridge Polynomials}
%%%%%%%%%%%%%%%%%%%%
\label{sec:hist} 

Some properties can be checked in a straightforward manner, such as the fact $S$ can be rewritten as

\[
S_{N,j}(z)=z+...+ (-1)^{j-1}z^N/N.
\]

Suffridge was the first to study the essential properties of these polynomials. 

In particular, he proved the univalency in $\mathbb D$, thus $S_{N,j}(z)$ are schlicht functions in  $\mathbb D$, i.e. univalent with zero coefficient zero and the degree one coefficient equal to one. 

Suffridge polynomials are also extremal:
since the derivative of a function univalent in $\mathbb D$ never vanishes in $\mathbb D$ the leading coefficient of the univalent polynomial of the degree $N$ cannot exceed  $1/N$ in absolute value.

In that sense Suffridge polynomial are on the verge of univalence - the roots of the derivative allowed on the boundary, which can seen from the image of the unit circle - it has cusps. 
This fact is illuminated on the hand drawn pictures below (borrowed from Suffridge's original paper~\cite{S}).

\begin{figure}[h!]
\centerline{
  \includegraphics[scale=0.15]{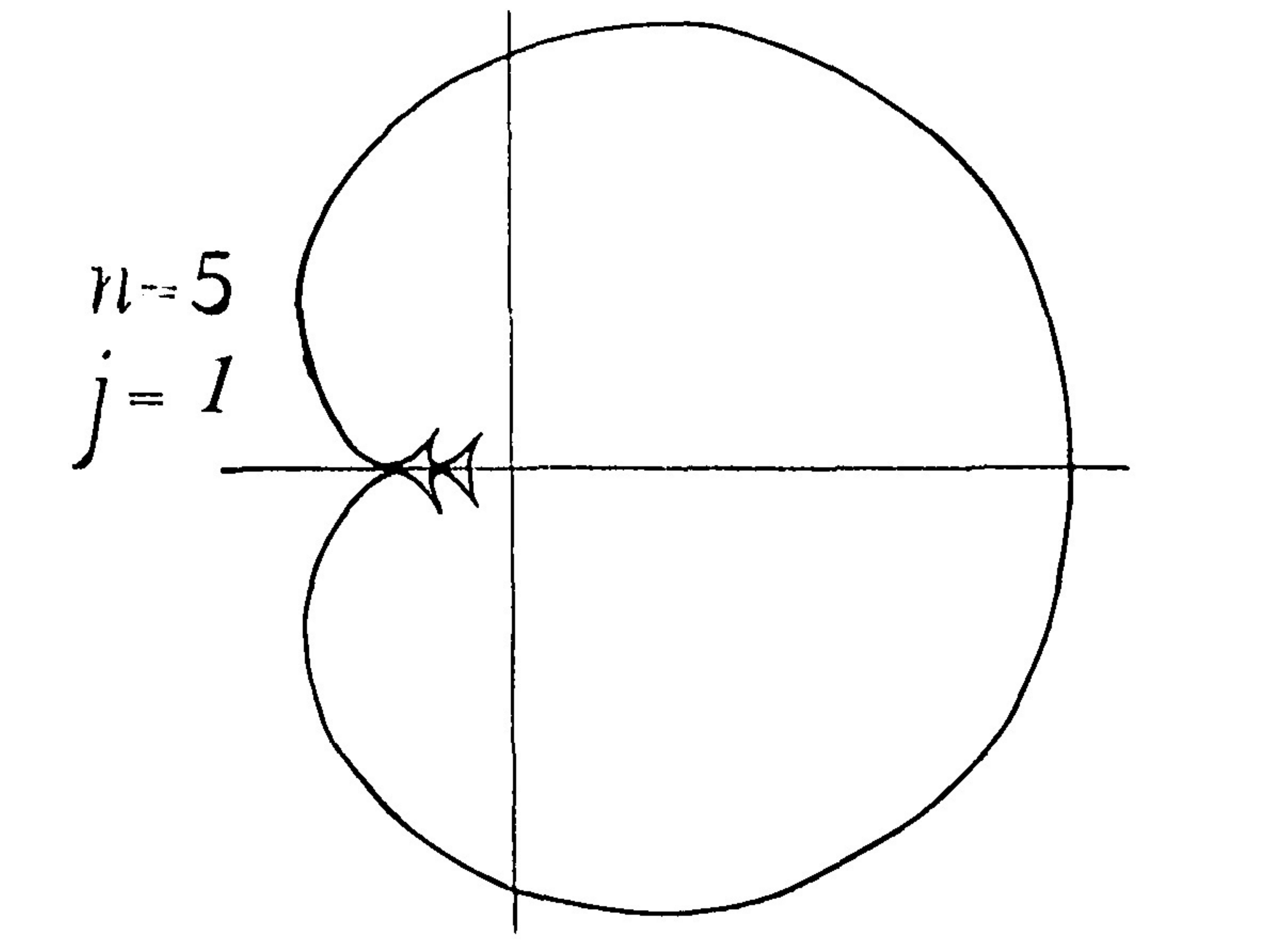}
}
  \caption{Image $S_{5,1}(\mathbb D)$, ex~\cite{S}}
  \label{fig:im1}
\end{figure}

\begin{figure}[h!]
\centerline{
  \includegraphics[scale=0.184]{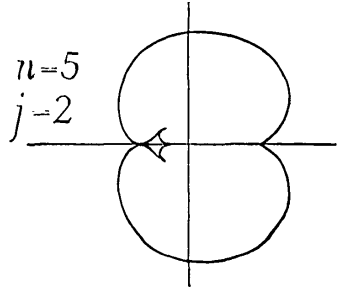}
 }
  \caption{Image $S_{5,2}(\mathbb D)$, ex~\cite{S}}
  \label{fig:im2}
\end{figure}

A careful observation of Figure~\ref{fig:im1} indicates that the the value $|S_{5,1}(-1)|$ might not be the minimal distance from the boundary of the region $S_{5,1}(\mathbb D)$ to the origin. 
More about this phenomena as well as some other interesting results can be found in~\cite{DDS}. 

Moreover, they are extremal as an existing substitute for Koebe functions in these settings.
The Koebe function is extremal in two famous theorems of geometric complex analysis:  the \emph{one quarter Koebe theorem} and in the \emph{Bieberbach conjecture} - now de Branges' theorem. 
Let us to remind these results\\

The {Koebe theorem} states that there exists a constant $r$, such that $f(\mathbb D)\supset D_r$  for any schlicht in $\mathbb D$  function $f.$ Here $D_r=\{|z|<r\}$. 
In 1916 Bieberbach proved that $r=1/4$ 
and that the extremal function is unique (up to rotation). 
The extremal function is called Koebe function
\[
K(z)=\frac z{(1-z)^2}=z+2z^2+3z^3+...
\]
It maps the unit disc $\mathbb D$ in the slit region 
$\mathbb C\backslash (-\infty,-1/4]$. 
In the same article Bieberbach proved that for any schlicht function $f(z)=z+a_2z^2+3z^3+...$  the estimate 
\[
|a_2|\le 2
\]
and conjectured that $|a_k|\le k, k=1,2,...$. 
For over 70 years this conjecture was a driving force behind the development of the geometric complex analysis. 
It was finally proved in 1984 by de Branges. 
Thus, it was established that the solution to both extremal problems function is the Koebe function $K(z)$.

It is natural to ask what is the polynomial version of these results. 
Note that the Koebe function is also on the verge of univalence - any increase in any of its coefficients eliminates the function from the  univalency class; it is thus natural to consider polynomials with the last coefficient $ %\pm 
1/N.$

Denote the coefficients of the Suffridge polynomial $S_{N, 1} (z) $ by $ A_j$, $j = 1, \ldots, N $.  
Let $ P_N (z) = z + a_2z ^ 2 + ... + a_Nz ^ N$ be any schlicht polynomial with real coefficients \footnote {Everywhere below we will assume all coefficients to be real} and the normalization to be as described above. Then 
$$ 
|a_j| \le A_j, j = 1, ..., N. 
$$

Further, Suffridge found an expression for the values of the polynomial on the unit circle (formula (5) in~\cite{S})
\begin{equation} \label{s5}
S_{N,j}(e^{it})=\frac{N+1}{2N(\cos t -\cos\alpha)} + i \frac{\sin t(1-(-1)^je^{i(N+1)t})}{2N(\cos t -\cos\alpha)},
\end{equation}
where $\alpha=j\pi/(N+1).$ 
Surprisingly, the compact form of the Suffridge polynomials has not been found yet.  Let us derive such a form below.

\section{Closed form} 
%%%%%%%%%%%
\label{sec:closed}

%Let us conclude this note by deriving a closed form for Suffridge polynomials. 

\begin{lemma}[Key ingredient]
\label{lem:key}
{\sl
Let $\displaystyle S_N= \sum_{k=1}^N \sin(k\alpha) z^k$ and $\displaystyle  C_N=  \sum_{k=1}^N \cos(k\alpha) z^k$ then 
\[S_N= z\frac{ \sin(\alpha) - \sin\left( (N+1) \alpha\right)z^N + \sin\left(N\alpha\right)z^{N+1}}{1-2\cos(\alpha)z+z^2}\]
and
\[C_N= z\frac{ \cos(\alpha) - z - \cos\left( (N+1) \alpha\right)z^N + \cos\left(N\alpha\right)z^{N+1}}{1-2\cos(\alpha)z+z^2}.\]
}
\end{lemma}

\begin{proof}[Proof]
By using the exponential notation, we can express $S_N$ and $C_N$ simultaneously as a finite geometric sequence
\[ C_N + i S_N =   \sum_{k=1}^N \left(e^{i\alpha} z\right)^k \]
which can then be rewritten as
\[ \frac{1-\left(e^{i\alpha}z \right)^N}{1-e^{i\alpha}z}e^{i\alpha}z . \]
The final result is then obtained by clearing the denominator of complex terms and using the real-imaginary decomposition.
\end{proof}

From this result we  derive (no pun intended) :

\begin{lemma}
\label{lem:dkey}
{\sl
Let $\displaystyle T_N= \sum_{k=1}^N k \sin(k\alpha) z^k$ then 
\[T_N = \frac{z}{(1-2\cos(\alpha)z+z^2)^2}\left\{
\sin (\alpha) \right. + \ldots \]
\[
-  \sin (\alpha ) z^2+ \ldots \]
\[
- ( N + 1 ) \sin ( (N+1) \alpha )  z^N + \ldots \]
\[
+( (N  + 2 )  \sin ( N \alpha ) 
+2 N \cos (\alpha ) \sin ( (N+1) \alpha ) ) z^{N+1} +\ldots \]
\[ 
- (( N -1)  \sin ( (N+1)\alpha ) 
+ 2 (N+1) \cos (\alpha ) \sin (N\alpha )) z^{N+2} + \ldots \]
\[\left.
+ N  \sin (\alpha  N) z^{N+3}\right\} 
\]
}
\end{lemma}

\begin{proof}[Proof]
This identity follows from Lemma~\ref{lem:key} and from the observation that $T_N= z \frac{d}{dz}S_N$.
\end{proof}

%%%%% \subsection{Suffridge polynomials.} %%%%%%%%%%%%

A first application of the above Lemmas is the derivation of a closed form for Suffridge polynomials.
We do not know of this form in the literature though Suffridge have used a similar argument to look at the image of the circle $\partial\mathbb D$. \\

\begin{proposition}[Suffridge in Closed Form]
\label{pro:closed}
{\sl
The Suffridge polynomials can be written as 
\[ S_{N,j}(z)= z\frac{ N-2(N+1)\cos\left(\frac{j\pi}{N+1}\right)z+(N+2)z^2+(-1)^j z^{N+1}-(-1)^j z^{N+3}}{N\left(1-2\cos\left(\frac{j\pi}{N+1}\right)z+z^2\right)^2} \]
}
\end{proposition}

\begin{proof}[Proof]
Recall that the Suffridge polynomials are defined as
\[S_{N,j}(z)=\sum_{k=1}^N A_k z^k\]
where
\[A_k=A_k(N,j)=\frac{N-k+1}{N}\frac{\sin\left(k\frac{j\pi}{N+1}\right)}{\sin\left(\frac{j\pi}{N+1}\right)}.\]
We can decompose $A_k$ as 
\[ \frac{N+1}{N}\frac{\sin\left(k\frac{j\pi}{N+1}\right)}{\sin\left(\frac{j\pi}{N+1}\right)}
- \frac{k}{N}\frac{\sin\left(k\frac{j\pi}{N+1}\right)}{\sin\left(\frac{j\pi}{N+1}\right)}. \]
Applying Lemmas~\ref{lem:key} and~\ref{lem:dkey} respectively to the first and second term gives the above result.
\end{proof}

\begin{example}
$S_{5,2}(z)
=\frac{z\left( 5-6z+7z^2- z^6+z^8
\right)}{5\left(1-z+z^2\right)^2}
=-\frac{z^5}{5}-\frac{2 z^4}{5}+\frac{4 z^2}{5}+z$
\end{example}

\section{Brandt representation}
%%%%%%%%%%%%%%%
\label{sec:brandt}

Note that Brandt~\cite{B2} suggested a general form of typically real polynomials $z+a_2z^2+...+a_nz^n$. The set of such polynomials was denoted by $T_n.$ 
His idea in itself is quite remarkable, but the coefficients included in it are difficult to choose. Here is a theorem from~\cite{B2}

\begin{theorem}
Let $f(z)$ be a rational function normalized by
\[
\label{eq:2}
f(0)=0,\qquad f^\prime(0)=1.%\leqno(2)
\]
Then the following statements are equivalent:

\begin{enumerate}[label=(\alph*)]

\item $f(z)$ belongs to the class $T_n$

\item There are $n$ real numbers $b_k$ and $b_k'$ with
\[
\label{eq:3}
f(z)=4n\sum_{k=1}^{[n/2]}\left(b_k^2\cos^2\frac{(2k-1)\pi}{2n} +
{b_k'}^2\sin^2\frac{(2k-1)\pi}{2n} \right)\frac{z}{1-2z\cos\frac{(2k-1)\pi}n+z^2}%\leqno(3)
\]
\[
-(1+z^n)(1-z^2)(1+z^2)\left(\sum_{k=1}^{[(n+1)/2]}
\frac{b_k}{1-2z\cos\frac{(2k-1)\pi}n+z^2} \right)^2
\]
\[
+(1+z^n)(1-z^2)(1+z^2)\left(\sum_{k=1}^{[n/2]}
\frac{b_k'}{1-2z\cos\frac{(2k-1)\pi}n+z^2} \right)^2
\]
\[
+\left(\sum_{k=1}^{[(n+1)/2]} b_k' \right)^2
-\left(\sum_{k=1}^{[n/2]} b_k' \right)^2
\]

\item There are $n$ real numbers $c_k$ and $c_k'$ with
\[
\label{eq:4}
f(z)=4n\sum_{k=1}^{[(n-1)/2]}\left(c_k^2\cos^2\frac{(k\pi}{n} +
{c_k'}^2\sin^2\frac{(k\pi}{n} \right)\frac{z}{1-2z\cos\frac{2k\pi}n+z^2}% \leqno(4)
\]
\[
-(1-z^n)(1-z^2)(1+z^2)\left(\sum_{k=1}^{[n/2]}
\frac{c_k}{1-2z\cos\frac{(k\pi}n+z^2} \right)^2
\]
\[
+(1-z^n)(1-z^2)(1+z^2)\left(\sum_{k=1}^{[(n-1)/2]}
\frac{c_k'}{1-2z\cos\frac{2k\pi}n+z^2} \right)^2
\]
\[
+\left(\sum_{k=1}^{[n/2]} c_k \right)^2
-\left(\sum_{k=1}^{[(n-1)/2]} c_k' \right)^2.
\]

\end{enumerate}
\end{theorem}

Since Suffridge polynomials are typically real,  they  should have an appropriate Brandt representation.  We managed to find the coefficients. 

Namely, to represent Suffridge polynomials in the form of rational function we can use formula (3) of Theorem 1 for odd $j$ and formula (4) if $j$ is even. 
Let 
\[
n=N+1,\quad k=\lfloor \frac{j+1}2\rfloor\quad\mbox{and}\;\quad b_k=b_k^\prime=\frac1{2\sqrt N},\;%(j - odd),  
c_k=c_k^\prime=\frac1{2\sqrt N}. %(j - even).
\]
Then it is easy to check the formulas (3) and (4) coincides with our Proposition 1.\\

\section{Extremal properties of Suffridge polynomials}
%%%%%%%%%%%%%%%%%%%%%%%%%%%
\label{sec:extremal}

Formula \eqref{s5} allows to compute the quantity
\[
S_{N,1}(-1)=-\frac14\frac{N+1}{N}\sec^2\frac\pi{2(N+1)}\to -\frac14,\; N\to\infty .
\] 
Thus Suffridge polynomials can be used to prove that the $\sfrac{1}{4}$ constant in Koebe's theorem is sharp. Hence, in the above sense, Suffridge polynomials can be considered as a substitute for Koebe functions.

By the way, Brandt~\cite[p.79 (466)]{B} solved the extreme problem of evaluating the modulus of schlicht polynomials of degree $N, $ showing that the extreme polynomial is a Suffridge polynomial, and that the maximum value is
\[
S_{N,1}(1)=\frac14\frac{N+1}{N}\csc^2\frac\pi{2(N+1)}.
\]
  
Dmitrishin and Khamitova~\cite{DK} announced a new non-obvious extreme properties of Suffridge polynomials. 
Namely, they are, after appropriate renormalization, the only optimal polynomials for the following extremal problem
\[
\sup_ {a_1+...+a_N=1} 
\left ({ \min_t 
\left\{ {\Re \left ({ F_N\left ({ { e^ {it} } } \right) } \right): \Im 
\left ({ F_N\left ({ { e^ {it} } } \right) } \right) = 0 } \right\} } \right) %= -\frac14 \sec ^2\frac\pi{N + 2}.  
\]
where $F_N(z)$ is any polynomial of degree $N$ with zero coefficient zero. 
The solution to this extreme problem as well very elegant and surprising applications to discrete dynamic systems are given in~\cite{DHKKS}.

Another nice feature of Suffridge polynomials was observed by Genthner, Ruscheweyh and Salinas.
In~\cite{GRS}, they proposed an interesting characterization of the boundary of simply connected domains. 
More precisely, an oriented closed curve $\gamma:[0,2\pi)\rightarrow{\mathbb C}$ is called quasi-simple if it represents the positively oriented boundary of a simply connected domain $D_\gamma\subset{\mathbb C}$.
In this paper the authors give a nonstandard criterion for closed plane curves to be quasi-simple. 
Along the way, Genthner et al. define the concept of quasi-extremal polynomial.

\begin{definition}[\cite{GRS}, Definition~11] \label{def:qe}
Let $P$ be a complex polynomial of degree $n$.
We call $P$ {\it quasi-extremal} (q-e) if there exists a simply connected domain $\Omega\subset{\mathbb C}$ such that
\begin{enumerate}
\item $P(\mathbb D)\subset \Omega$
\item $P'$ has $n-1$ zeroes on $\partial \mathbb D$, say in $e^{i\theta_k}$, where the angles $\theta_k$ are labeled such that $\theta_1< \theta_2<\ldots <\theta_n=\theta_1+2\pi$.
\item There exists $\tau_j \in [\theta_j,\theta_{j+1})$ with $P(e^{i\tau_j})\in \partial \Omega$, $j=1,\ldots,n-1$.
\end{enumerate}
\end{definition}

The pertinence of q-e polynomials is a consequence of the following theorem.

\begin{theorem}[\cite{Gen,GRS}]
Every quasi-extremal polynomial is univalent in $\mathbb D$.
\end{theorem}

Quasi-extremal polynomials are interesting candidates as extremal polynomials for maximal range problems - see~\cite{CR1,CR2,4}.
In particular, Suffridge polynomials are a particular example of q-e polynomials as one can observe from their boundary representation.
%{\color{blue} 
The disposition of the roots of the derivatives of polynomials is an important subject in the theory, see
for example the recent article by N. Naidenov, G. Nikolov, and A. Shadrin \cite{NNS}.
%}

\section{Approximation of the Koebe function}
%%%%%%%%%%%%%%%%%%%%%%%
\label{sec:approx}

In \cite{D} it was shown by Dimitrov that 
\[
\displaystyle\left(-\frac14,\frac14\cot^2\frac\pi{2N+2}\right)\subset\frac{-1}{4S_{N,1}(-1)}S_{N,1}(\mathbb D).
\] 
Now, in \cite{AR} Andrievskii and Ruscheweyh proved a remarkable theorem on polynomial approximation of conformal mappings of the unit disc $\mathbb D.$ 
According to their key theorem, there exists a constant $\rho_N$ such 
\[
\displaystyle K(\rho_N\mathbb D)\subset\frac{-1}{4S_{N,1}(-1)}S_{N,1}(\mathbb D)
\]
where the constant $\rho$ is computed in \cite{GR}:
\[
\rho_N=\frac{1-\sin\frac\pi{2N+2}}{1+\sin\frac\pi{2N+2}}\sim 1-\frac\pi N.
\]
This proves that Suffridge polynomials $S_{N,1}(z)$ approximate Koebe function in a sense or subordination of the images of $\mathbb D.$\\

Let us now consider a problem of uniform approximation of Koebe function by Suffridge polynomials in the disc $D_\rho=\{z:|z|<\rho\}.$ 
To this effect, let us modify the close form of these polynomials from the Proposition~\ref{pro:closed}.
Let us restrict the value of $j$ to
\[
\frac j N=O\left(\frac1N\right).
\]
Using quadratic approximation of the cosine and the geometric series decomposition we get 
\[
\frac z{1-2\cos\left(\frac{j\pi}{N+1}\right)z+z^2}=K(z)+O\left(\frac1{N^2}\right),
\]
\[
\frac zN \frac{ 2z\left(z- \cos\frac{j\pi}{N+1}\right) + (-1)^j z^{N+1}(1-z^2)}{\left(1-2\cos\left(\frac{j\pi}{N+1}\right)z+z^2\right)^2}=
\frac{-2z^2+ (-1)^j z^{N+1}(1+z)}{(1-z)^3}+O\left(\frac1{N^3}\right)
\]
From there
\[
\|S_{N,j}(z)-K(z)\|<\frac2{(1-\rho)^3}\frac1N+o\left(\frac1N\right).
\]

Now, let  $N$ be odd and $j=\frac{N+1}2$; then the polynomials $S_{n,\frac{N+1}2}(z)$ can be used to approximate the 2-symmetric Koebe function $K^{(2)}(z)=\frac z{1+z^2},$ and
\[
\|S_{n,\frac{N+1}2}(z)-K^{(2)}(z)\|<\frac1{2(1-\rho)^2}\frac1N+o\left(\frac1N\right).
\]
Finally, if $j=qN,$ % i.e. \textcolor{blue}{$q$ is a divisor of $N$}, 
then
\[
\|S_{N,j}(z)-\frac z{1-2z\cos q\pi+z^2}\|<\frac1{2\sin q\pi(1-\rho)^2}\frac1N+o\left(\frac1N\right).
\]

These estimates demonstrate that the approximation of univalent functions by univalent polynomials in $D_\rho$ might have a small rate of convergence with respect to the uniform metric due to the large constants depending on $\rho$ while the approximation in sense of subordinations is admissible. 

As a final remark, note that  the coefficients of the polynomials  $S_{N,1}(z)$ and $S_{N,\frac{N+1}2}(z)$ also appear in the problem of chaos control in the systems with discrete time \cite{DHKKS}.

\section{Robertson measure for Suffridge polynomials}
%%%%%%%%%%%%%%%%%%%%%%%%%%%
\label{sec:robertson}

The generating function for the Chebyshev polynomials of the second kind is the following
\[
\frac1{1-2tz+z^2}=\sum_{k=1}^\infty U_{k-1}(t)z^{k-1}
\]
These polynomials are orthogonal on the segment [0,1] with respect to the weight $\sqrt{1-t^2}$
\[
\frac2\pi\int_{-1}^1\sqrt{1-t^2}U_m(t)U_n(t)\; dt=\delta_{m,n}.
\]
Therefore
\[
z^{k-1}=\frac2\pi\int_{-1}^1 \frac1{1-2tz+z^2}\sqrt{1-t^2}U_{k-1}\; dt,\qquad k=1,2,...
\]
and an arbitrary polynomial $f(z)=\sum_{j=1}^Nb_jz^j$ admits a representation

\begin{equation} \label{eq:frep}
f(z)=\frac2\pi\int_{-1}^1 \frac1{1-2tz+z^2}\sqrt{1-t^2}
\sum_{k=1}^Nb_k U_{k-1}(t)\; dt. 
\end{equation}

Conversely, the function
\[
\frac2\pi\int_{-1}^1 \frac1{1-2tz+z^2}\sqrt{1-t^2}
\left(
\frac{c_0}2+\sum_{k=1}^{N-1}c_kt^k
\right)  dt
\]
is a polynomial
\[
f(z)=\frac12\sum_{j=1}^N(c_{j-1}-c_{j+1})z^j
\]
assuming that $c_N=c_{N+1}=0.$\\

Let us write~(\ref{eq:frep}) in the form 

\begin{equation} \label{eq:fint}
f(z)=\int_{-1}^1 \frac1{1-2tz+z^2}d\mu(t)
\end{equation}

where
\[
\mu(t)=\frac2\pi\int_{-1}^t \sqrt{1-\xi^2} \sum_{k=1}^Nb_k U_{k-1}(\xi)\;d\xi.
\]
Let us compute the integrals.
\[
\frac2\pi\int_{-1}^t \sqrt{1-\xi^2} U_0(\xi)\;d\xi=-\frac12+\frac1\pi\arcsin t+\frac1\pi t \sqrt{1-t^2}
\]
\[
\frac2\pi\int_{-1}^t \sqrt{1-\xi^2} U_{k-1}(\xi)\;d\xi=
-\frac1\pi t \sqrt{1-t^2}\left(\frac{U_{k-2}(t)}{k-1} -\frac{U_{k}(t)}{k+1}\right)
,\qquad k=2,...,N.
\]
Note \cite{DDS} that
\[
(k+1)U_{k-2}(t)-(k-1)U_k(t)=2(1-t^2)U^\prime_{k-1}(t). 
\]
Therefore
\[
\frac2\pi\int_{-1}^t \sqrt{1-\xi^2} U_{k-1}(\xi)\;d\xi=\frac2\pi\frac{1-t^2}{k^2-1}U^\prime_{k-1}(t),\qquad k=2,...,N.
\]
Then
\[
d\mu(t)=\frac1\pi\sqrt{1-t^2}\left(1+3t\sum_{k=2}^Nb_k\frac1{k^2-1}U^\prime_{k-1}(t)-(1-t^2)\sum_{k=2}^Nb_k\frac1{k^2-1}U^{\prime\prime}_{k-1}(t)
\right)dt. 
\]
Because
\[
d\mu(t)=\frac2{\pi N}\sqrt{1-t^2}\sum_{k=2}^N b_kU_{k-1}(t) dt
\]
we obtain the well-known relation for Chebyshev polynomials of the second kind
\[
U_k(t)=\frac1{k(k+1)}\left( 3tU^\prime_k(t)-(1-t^2)U^{\prime\prime}_{k}(t)\right).
\]

If the polynomial $f(z)$ is typically-real then $\Im(f(e^{i t}))\ge 0$ for $t\in[0,\pi],$ this corresponds the inequality
\[
\sum_{k=1}^N b_kU_{k-1}(t)\ge 0 \qquad \mbox{for}\quad t\in [-1,1].
\]
In other words, the function $\mu(t)$ can be consider as a Robertson measure \cite{R} in the representation~(\ref{eq:fint}) for the typically real function.\\

Let us consider the Robertson measure $\mu$ for the polynomials $S_{N,j}(z)$. 
Taking into account that the coefficients of those polynomials are 
\[
a_k=\left(1-\frac{k-1}N\right)U_{k-1}\left(\cos\frac{\pi j}{N+1}\right),
\]
we get
\[
\mu(t)=-\frac12+\frac1\pi\arcsin t+\frac1\pi t \sqrt{1-t^2} 
-\frac2{\pi N}\sqrt{1-t^2}\sum_{k=2}^N\frac{N-k+1}{k^2-1}
U_{k-1}\left(\cos\frac{\pi j}{N+1}\right)U^\prime_{k-1}(t).
\]

The graphs  of Robertson measure for various Suffridge polynomials are given in Figure~\ref{fig:robertson}.

\begin{figure}[h!]
\centerline{
\includegraphics[scale=0.35]{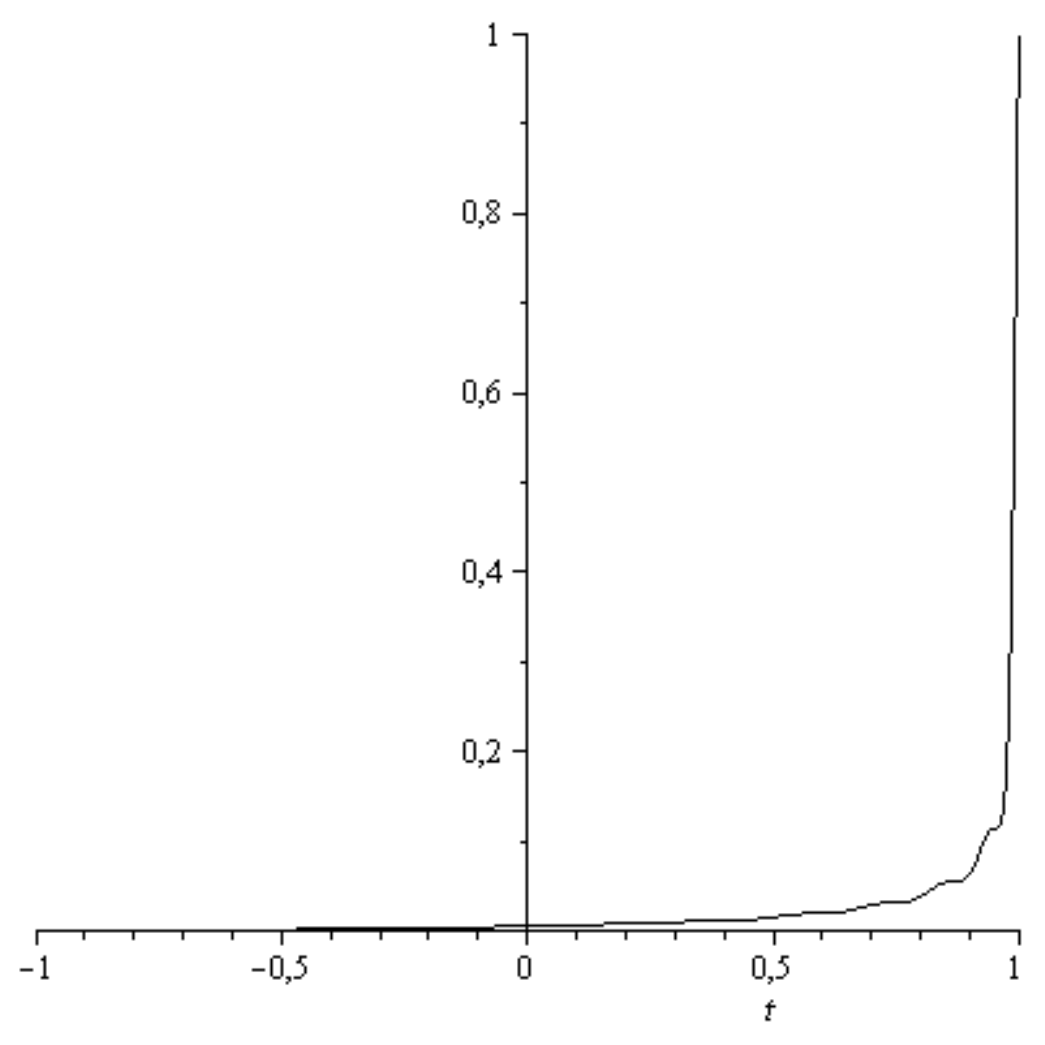}
\includegraphics[scale=0.35]{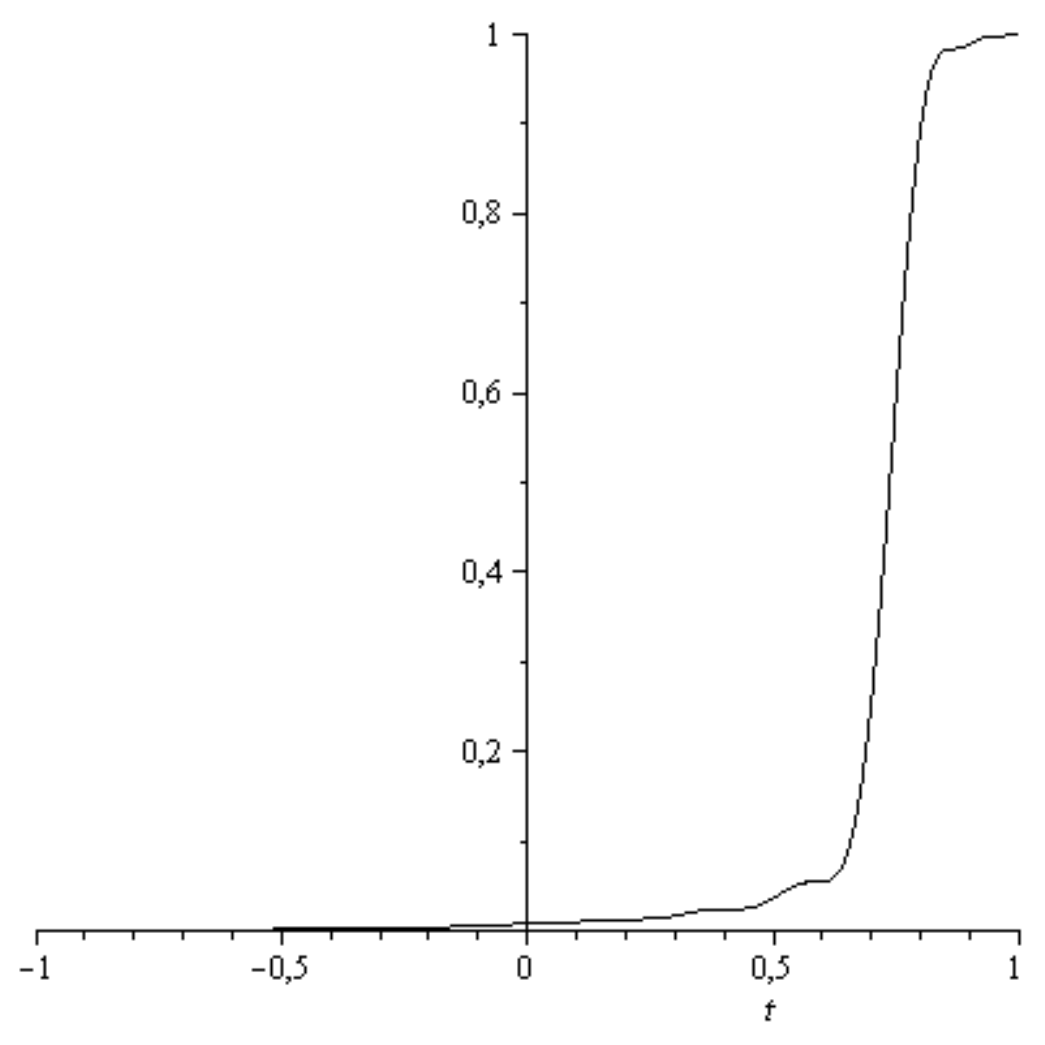}
\includegraphics[scale=0.35]{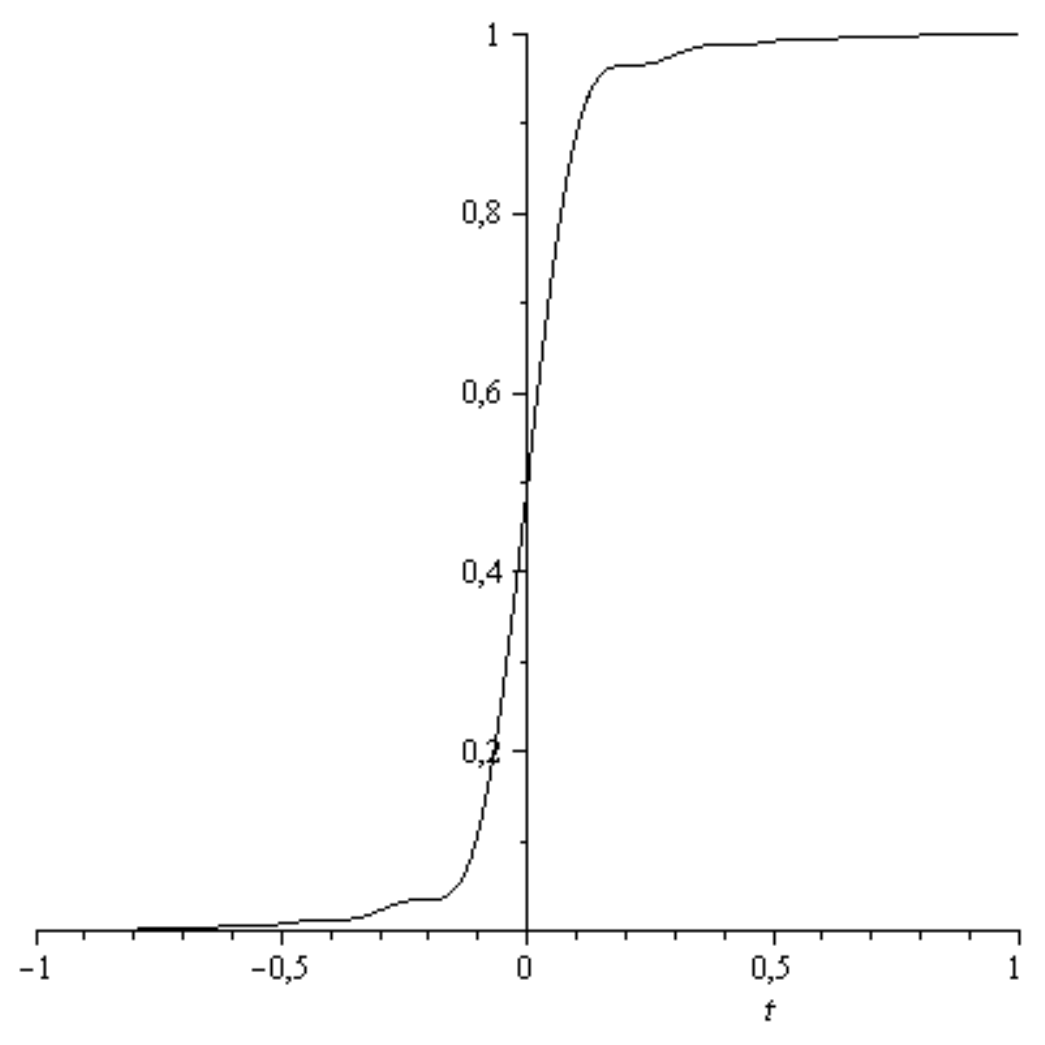}
}
\caption{{T}he graphs of Robertson measure for the polynomials $S_{29,1}(z)$, $S_{29,7}(z)$, $S_{29,15}(z)$}
\label{fig:robertson}
\end{figure}

Note that the Robertson measure approaches the step function
\[
\mu_\infty(t)= 
\begin{cases} 0, & \mbox{if } t\in\left.\left[-1,\cos\frac{\pi j}{N+1}\right)\right. \\ 
1, & \mbox{if } t\in\left[\cos\frac{\pi j}{N+1},1\right]
\end{cases},
\]
this can be considered as a monotonic approximation of step functions. 
Alternative kernels and another measures, including ones for Suffridge polynomials, can be found in \cite{GH}.

\section{Robust Univalency}
%%%%%%%%%%%%%%
\label{sec:robust}

In addition to Suffridge polynomials $S_{N,j}(z)$  let us consider the polynomials 

\begin{equation}
\label{eq:gensuf}
S_{N}(z,\mu)= \sum_{k=1}^N \frac{N+1-k}{N} \frac{\sin(k\mu)}{\sin\mu} z^k, %\eqno(8.1)
\end{equation}

where $\mu$ is a real parameter. It is clear that $S_N\left(z, \frac{j\pi}{N+1}\right)=S_{N,j}(z)$ and that without loss of generality we can restrict ourselves to the case $\mu\in (0,\pi).$ If 
\[
\mu\in\left\{\frac{j\pi}{N+1}: j=1,...,N\right\}
\]
then the polynomials~(\ref{eq:gensuf}) are univalent. Let us ask a question: {\it are there values for 
\[
\mu\in(0,\pi)\backslash\left\{\frac{j\pi}{N+1}: j=1,...,N\right\}
\]
for which the polynomials~(\ref{eq:gensuf}) are univalent for some $N$.}

Using Lemmas~\ref{lem:key} and~\ref{lem:dkey} we can write
\[
S_N(z,\mu)=\frac z{(1-2z\cos\mu+z^2)^2}\left( 
1-2z(1+\frac1N)\cos\mu + (1+\frac2N)z^2+
\right.
\]
\[
\left.+\frac{z^{N+1}}{N\sin\mu}(\sin(N+2)\mu -2z\sin(N+1)\mu+z^2\sin N\mu)\right)
\]
or
\[
S_N(z,\mu)=\frac z{(\frac1z-2\cos\mu+z)^2}\left(\frac1z+z-2(1+\frac1N)\cos\mu+
\right.
\]
\[
\left.
\frac1{N\sin\mu}(2z\sin\mu+z^N\sin(N+2)\mu-2z^{N+1}\sin(N+1)\mu +z^{N+2}\sin N\mu)
\right).
\]
Therefore, the sign of $\Im(S_N(e^{it},\mu))$ coincides  with the sign of the function
\[
\Phi(t,\mu)=2\sin t\sin \mu+\sin Nt\sin(N+2)\mu-2\sin(N+1)t\sin(N+1)\mu+\sin(N+2)t \sin N\mu.
\]
However, 
\[
\Phi(\mu+\frac{2\pi}{N+1},\mu)=\Phi(\mu-\frac{2\pi}{N+1},\mu)=-4\sin^2\frac\pi{N+1}\sin^2(N+1)\mu.
\]

Thus, for any 
\[
\mu\in(0,\pi)\backslash\left\{\frac{j\pi}{N+1}: j=1,...,N\right\}
\]
there exists a number $t\in (0,\pi)$, more precisely $t=\mu\pm\frac{2\pi}{N+1}$ such that $\Phi(t,\mu)<0$. 
That means that all polynomials of the form~(\ref{eq:gensuf}), except for the Suffridge ones, are not typically real, hence not univalent. 
In that sense, the robust stability of the polynomials from the family~(\ref{eq:gensuf}) is false.

\section{A new family of polynomials}
%%%%%%%%%%%%%%%%%%%
\label{sec:new}

At a first glance, it seems that the lack of the robust univalency of the family of polynomials~(\ref{eq:gensuf}) is a consequence of the quasi-extremality in the Ruscheweich sense, i.e. the polynomials are on the verge of univalency -- all zeroes of the derivative of those polynomials are on the unit circle (see the Definition~\ref{def:qe} above). f
Instead, let us consider the following family
\begin{equation}
\label{eq:G}
G_{N,\mu}(z)=z+\sum_{k=2}^N\left(1-\frac{k-1}N\right)\prod_{j=1}^{k-1}\frac{\sin\frac{\pi(j+\mu)}{N+\mu}}{\sin\frac{\pi j}{N+\mu}}z^k.
\end{equation}
Note that for $\mu=1$ the polynomial $G_{N,\mu}(z)$ coincides with $S_{N,1}(z),$ for $\mu=0$ it is the Fej\'er polynomial, and 
\[
G_{N,-1}(z)=z+\frac{z^N}N.
\]
One can see from Figure~\ref{fig:G} that the polynomials $G_{N,\mu}(z)$ look like univalent q-e polynomials.\\

{\bf Conjecture} {\it There exists a constant $\zeta(N)<-1$ such that all polynomials of the family~(\ref{eq:G}) are univalent in the unit disc $\mathbb D$ for $\mu\in[-\zeta(N),1].$ Moreover, all zeroes of the derivative of the polynomials~(\ref{eq:G}) are on the unit circle and the segment $[-\zeta(N),1]$ cannot be enlarged without loss of univalency.}\\

In other words the Conjecture claims that all polynomials $G_{N,\mu}(z)$ are quasi-extremal in the Ruscheweich sense.\\

Note that in a central disc of radius smaller then 1, the polynomials $G_{n,-1}(z)$ approximate the function $\frac z{(1-z)^{1+\mu}}.$
Moreover, it is interesting to note that if one drops sine symbols in~(\ref{eq:G}) then the product turns into a normalized binomial coefficient $\displaystyle {\mu+k-1\choose k-1}.$ 
In general one could replace sine with any function $g(t)$ and thus have $g$-generalized binomial coefficients and the corresponding polynomials.\\

\section{Conclusion}
%%%%%%%%%%

In the article we have summarized some key facts about Suffridge polynomials that do not appear in the literature and which seem quite relevant and interesting given the wide scope of applications for these polynomials. 
One of the next steps to undertake is to better understand and hopefully provide a proof of the conjecture stated in the previous section. 

\begin{figure}[t]
    \centering
    \begin{subfigure}[b]{0.3\textwidth}
        \includegraphics[width=\textwidth, scale=0.80]{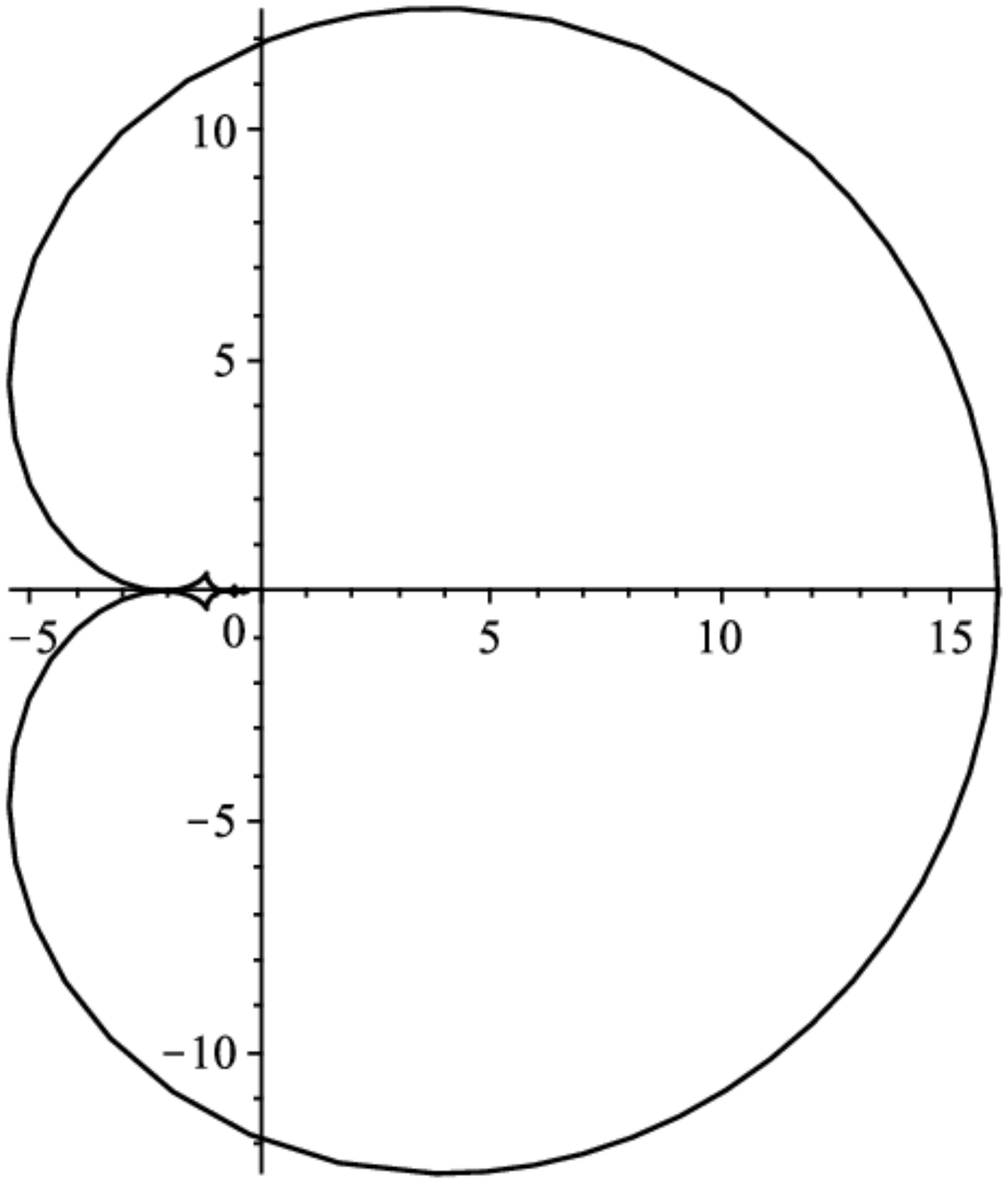}
    \end{subfigure}
    \begin{subfigure}[b]{0.3\textwidth}
        \includegraphics[width=\textwidth, scale=0.80]{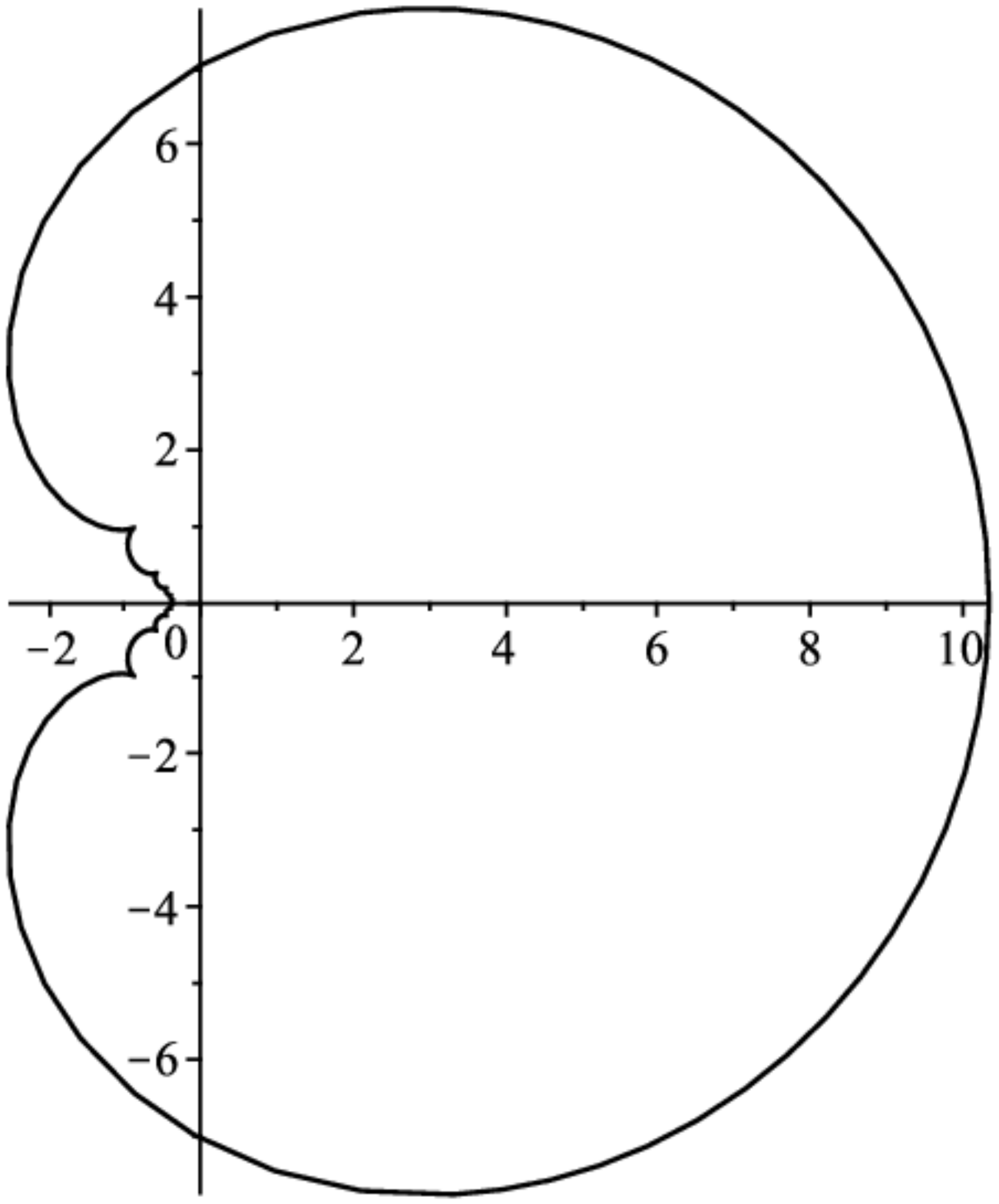}
    \end{subfigure}
    \begin{subfigure}[b]{0.3\textwidth}
        \includegraphics[width=\textwidth, scale=0.80]{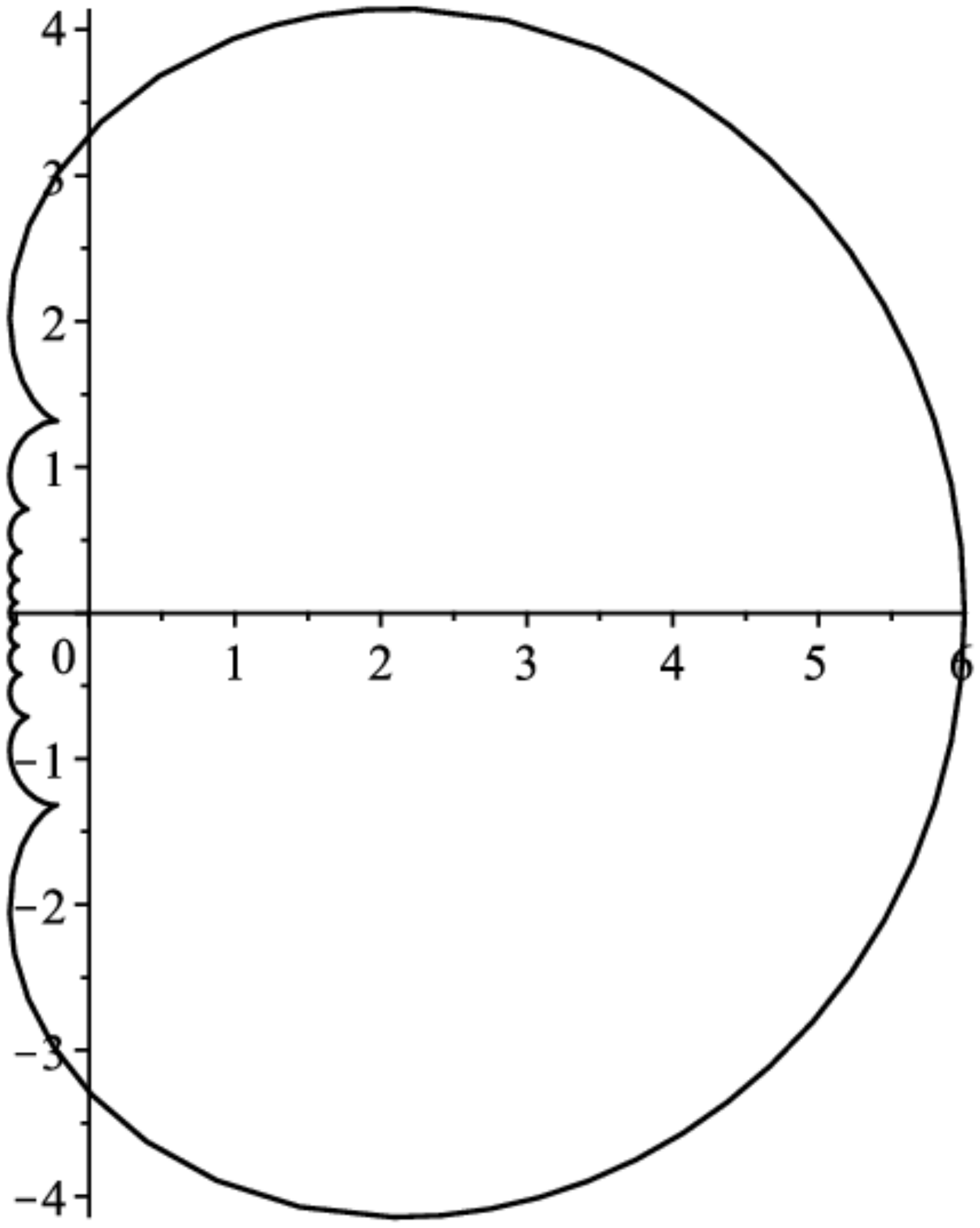}
    \end{subfigure}
     \begin{subfigure}[b]{0.3\textwidth}
        \includegraphics[width=\textwidth, scale=0.80]{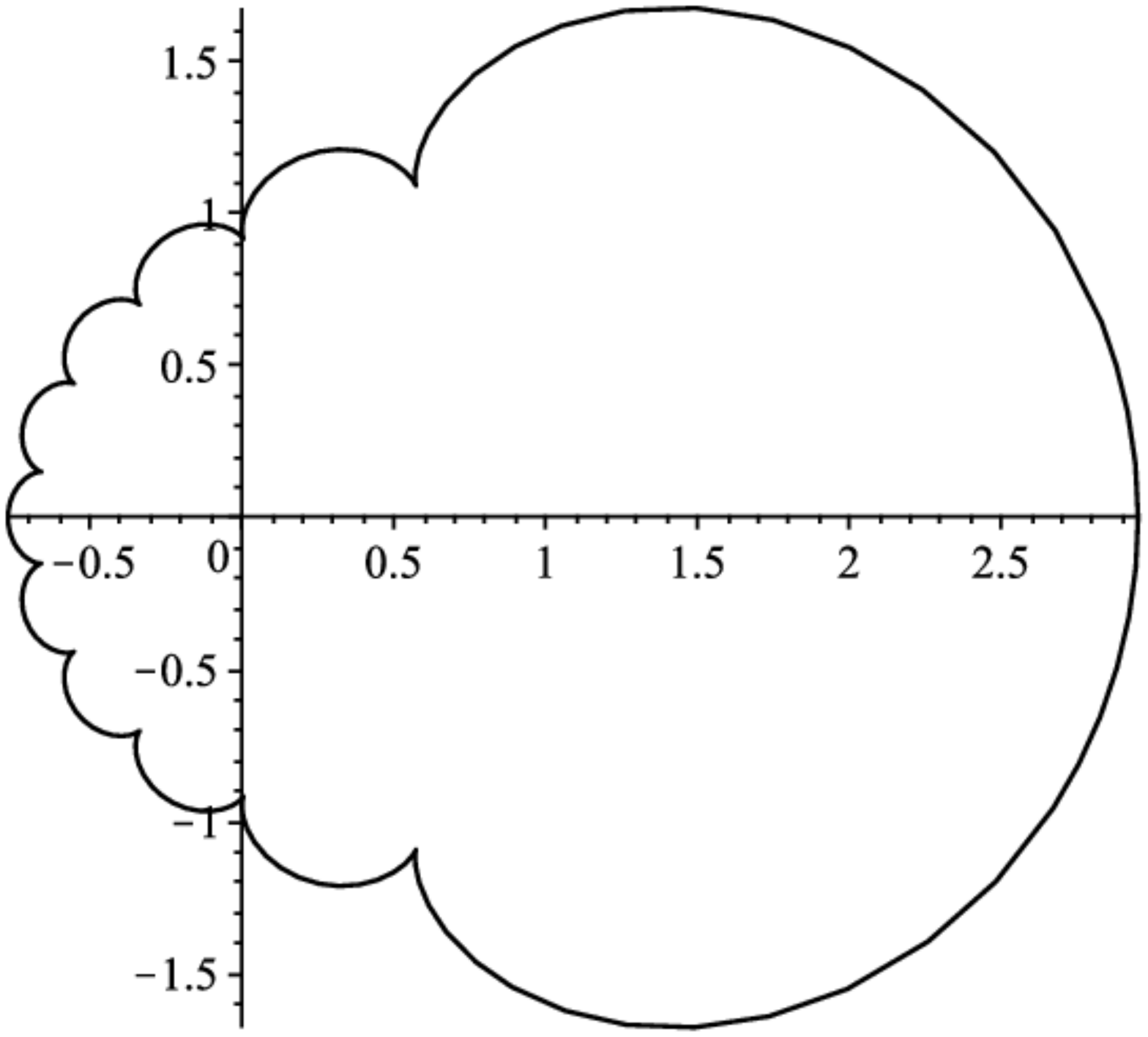}
    \end{subfigure}
    \begin{subfigure}[b]{0.3\textwidth}
        \includegraphics[width=\textwidth, scale=0.80]{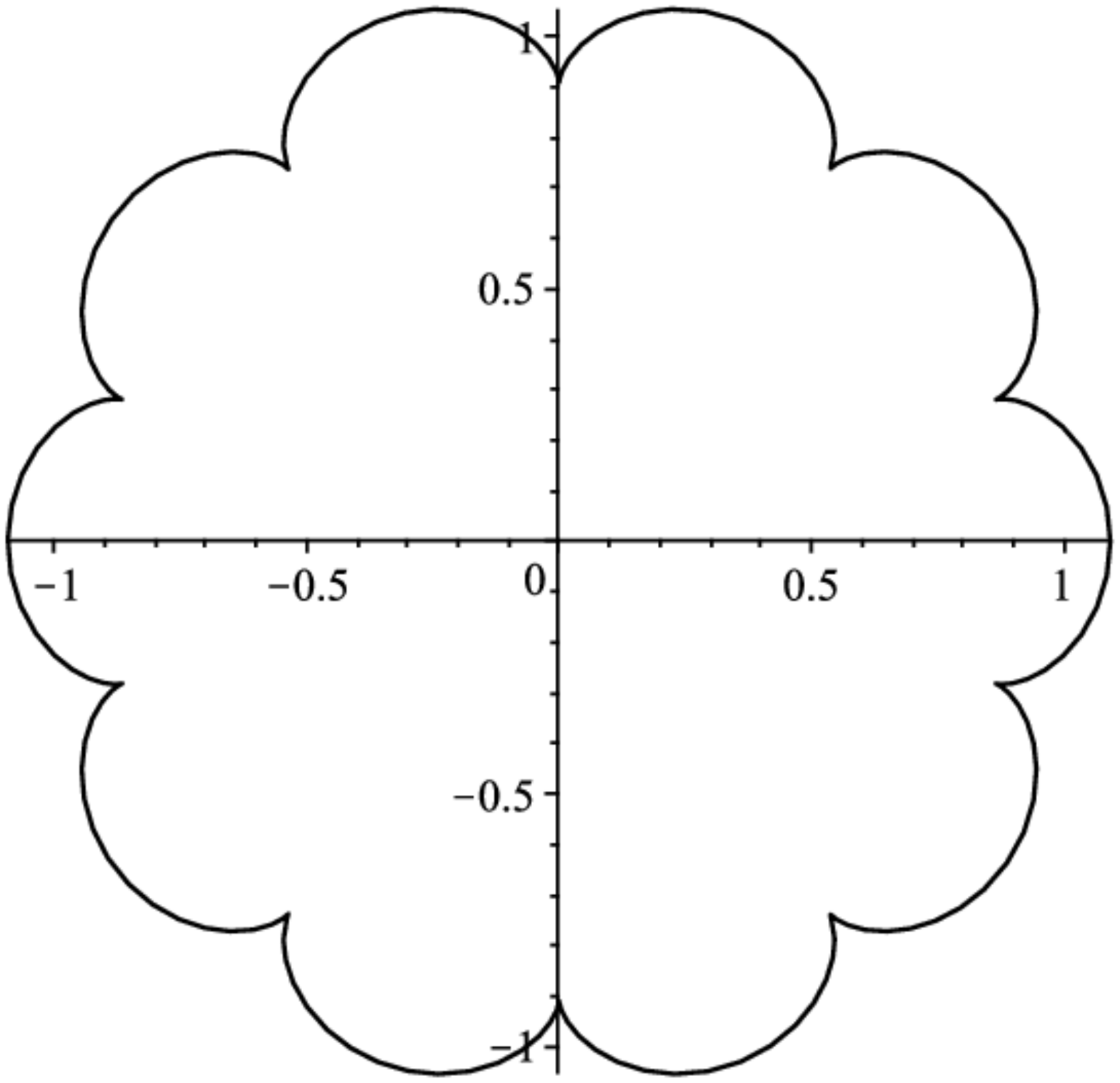}
    \end{subfigure}
    \begin{subfigure}[b]{0.3\textwidth}
        \includegraphics[width=\textwidth, scale=0.80]{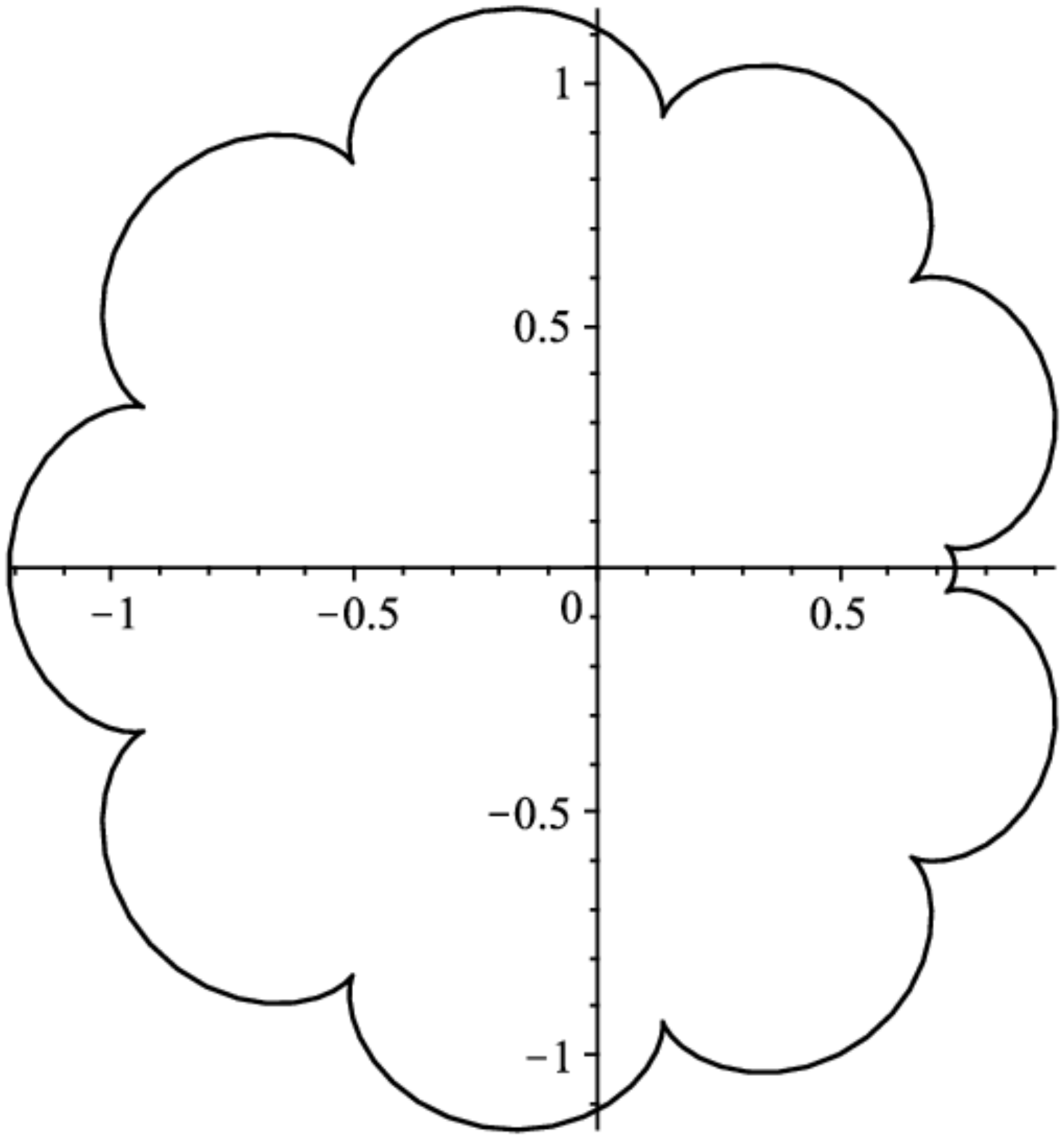}
    \end{subfigure}
    \caption{Images of the unit circle under polynomial mappings $G_{11, \mu}(z)$: a) $\mu = 1$; b) $\mu = 0.5$; c) $\mu = 0$; d) $\mu = -0.5$; e) $\mu = -1$; f) $\mu = -1.15$.}
    \label{fig:G}
\end{figure}

\end{document}